\documentclass{amsart}

\usepackage{mathtools}
\usepackage{enumitem}
\usepackage{amssymb}

\newtheorem{theorem}{Theorem}[section]
\newtheorem{lemma}[theorem]{Lemma}
\newtheorem{proposition}[theorem]{Proposition}

\theoremstyle{definition}
\newtheorem{definition}[theorem]{Definition}

\theoremstyle{remark}
\newtheorem{remark}[theorem]{Remark}

\DeclareMathOperator{\ad}{ad}
\DeclareMathOperator{\Cl}{Cl}
\DeclareMathOperator{\dom}{dom}
\DeclareMathOperator{\OP}{OP}
\DeclareMathOperator{\sgn}{sgn}
\DeclareMathOperator{\Tr}{Tr}
\DeclareMathOperator{\Ree}{Re}
\DeclareMathOperator{\Imm}{Im}

\numberwithin{equation}{section}

\begin{document}

\title{Pseudo-Riemannian Spectral Triples for $\mathrm{SU}(1,1)$}


\author{Jort J.A.M. de Groot}
\address{Korteweg–de Vries Institute voor Wiskunde, University of Amsterdam, Sciencepark 105-107, 1090
GE Amsterdam, The Netherlands}
\email{j.j.a.m.groot@uva.nl}



\date{2026}

\dedicatory{}

\begin{abstract}
    We use the harmonic analysis of $\mathrm{SU}(1,1)$ to show that the triple $(\mathcal{A},\mathcal{H},D)$, with $D$ (the closure of) Kostant's cubic Dirac operator acting on the Hilbert space $\mathcal{H}=L^2(\mathrm{SU}(1,1))\otimes\mathbb{C}^2$, and with $*$-algebra $\mathcal{A}=C^\infty_c(\mathrm{SU}(1,1))\otimes 1$, forms both a pseudo-Riemannian spectral triple in the sense of Van den Dungen, Paschke and Rennie, and an indefinite spectral triple in the sense of Van den Dungen and Rennie.
\end{abstract}

\maketitle

\section{Introduction}
One of the biggest accomplishments of the noncommutative geometry programme is the reconstruction theorem by Alain Connes \cite{connes_2013}, which states that to a spectral triple $(\mathcal{A},\mathcal{H},D)$ with $\mathcal{A}$ commutative, satisfying five further conditions, is uniquely associated a smooth oriented compact (spin$^c$) manifold $X$ with $\mathcal{A}=C^\infty(X)$ and with Riemannian metric constructed from $D$. As of yet, however, it remains unknown how to prove a similar theorem for manifolds that are not necessarily compact and that are equipped with a pseudo-Riemannian metric. Hence, it also remains an open question what the correct formulation of a spectral triple should be for those manifolds. 

There have been multiple attempts to formulate spectral triples in different settings. An incomplete list of examples includes: pseudo-Riemannian spectral triples (à la Strohmaier) \cite{strohmaier_2006}, equivariant Lorentzian spectral triples \cite{paschke_sitarz_2006}, semifinite spectral triples \cite{carey_gayral_rennie_sukochev_2014}, Krein spectral triples \cite{dungen_2016}, pseudo-Riemannian spectral triples (à la Van den Dungen, Paschke, Rennie) \cite{dungen_paschke_rennie_2013}, and indefinite spectral triples \cite{dungen_rennie_2016, dungen_2019}. The latter pseudo-Riemannian spectral triple is a generalization of semifinite and equivariant Lorentzian spectral triples, and is compatible with the framework developed for semifinite spectral triples (especially the local index theorem) \cite{dungen_paschke_rennie_2013}. On the other hand, Krein spectral triples form a complementary viewpoint. Indefinite spectral triples are on a similar footing to the later pseudo-Riemannian spectral triples. In fact, the two definitions are closely related and might even be equivalent, although a proof of that is still missing.

Typically, the Dirac operator on a pseudo-Riemannian manifold is not symmetric. Hence, in the assumptions for spectral triples that describe pseudo-Riemannian manifolds, it is no longer required that the Dirac operator is essentially self-adjoint. Instead, the Dirac operator $D$ is supposed to be closed and densely defined. The definition of a pseudo-Riemannian spectral triple \cite{dungen_paschke_rennie_2013} contains assumptions on second-order operators, that ensure that there is an associated spectral triple, obtained via Wick rotation. With this, one obtains a $K$-homology class so that the local index formula can be applied. For the Wick rotation, however, there are two possibilities. On the other hand, the definition of indefinite spectral triples \cite{dungen_rennie_2016,dungen_2019} contains assumptions on first-order operators, and associates \emph{two} spectral triples to the indefinite spectral triple. Moreover, this process can be reversed, and out of the two spectral triples the indefinite spectral triple can be reconstructed.

The motivating example for both pseudo-Riemannian spectral triples and indefinite spectral triples is given by pseudo-Riemannian spin manifolds $(M,g)$ for which the tangent bundle splits as $TM=E_t\oplus E_s$, with $g$ negative definite on $E_t$ and positive definite on $E_s$. This allows to define a reflection $r\colon TM\rightarrow TM$ by $r\coloneq 1-2T$, with $T$ the projection $TM\rightarrow E_t$, and a Riemannian metric $g_E(v,w)\coloneq g(rv,w)$ (where $v,w\in TM$). If the resulting Riemannian manifold is complete and of bounded geometry, and moreover its Dirac bundle has bounded geometry, then the geometric Dirac operator $D$ arising from the metric $g$ determines a pseudo-Riemannian spectral triple \cite[Proposition 3.8]{dungen_paschke_rennie_2013}, and an indefinite spectral triple \cite[Proposition 4.5]{dungen_2019}. Moreover, in both cases, the assumptions on the spectral triples ensure that the combinations of the real and imaginary part of the Dirac operator, $D_\pm=\Ree D\pm \Imm D$, determine a unital spectral triple in the sense of \cite{carey_gayral_rennie_sukochev_2014}. Here $\Ree D=\tfrac{1}{2}(D+D^*)$ and $\Imm D=-\tfrac{i}{2}(D-D^*)$, where $D^*$ is the Hilbert space adjoint of $D$.

The aim of this paper is to construct both a pseudo-Riemannian spectral triple and an indefinite spectral triple associated to the Lie group $\mathrm{SU}(1,1)$. Although the manifold underlying $\mathrm{SU}(1,1)$ can be shown to satisfy the above requirements for the constructions in \cite{dungen_paschke_rennie_2013} and \cite{dungen_rennie_2016,dungen_2019}, the present paper makes the following points:
\begin{enumerate}
    \item[1)] The harmonic analysis related to the group structure of $\mathrm{SU}(1,1)$ can be used to approach the axioms for the spectral triples. This is not taken into account in the geometric construction, which uses only the manifold structure. Moreover, a proof relying on the harmonic analysis of $\mathrm{SU}(1,1)$ can more easily be adapted to the quantum group setting\footnote{This is a direction that is pursued in upcoming work, for the locally compact quantum group $\mathrm{SU}_q(1,1)\rtimes\mathbb{Z}_2$.}.
    \item[2)] The Wick rotation defined using the reflection $r$ has a natural interpretation in terms of the Cartan involution on $\mathrm{SU}(1,1)$. To be precise, the pseudo-Riemannian metric obtained from the Killing form $B$ is related to a Riemannian metric via the standard Cartan involution $\theta$ of $\mathfrak{su}(1,1)$, where the Riemannian metric is defined using the bilinear form $B_\theta(X,Y)\coloneq B(X,\theta Y)$ for $X,Y\in\mathfrak{su}(1,1)$.
    \item[3)] Instead of the geometric Dirac operator, we use another Dirac operator for $\mathrm{SU}(1,1)$ which naturally arises, namely Kostant's cubic Dirac operator \cite{kostant_1999}, which squares to the Casimir operator modulo a constant. In fact, as will become clear, for the construction of a pseudo-Riemannian spectral triple, the cubic Dirac operator has the advantage that the second-order operator $R_D=\tfrac{i}{2}(D^2-D^{*2})$ vanishes, which greatly simplifies the smoothness assumptions for the spectral triple. Moreover, for the construction of an indefinite spectral triple, working with the cubic Dirac operator has the advantage that the real and imaginary part of $D$ anti-commute, which is stronger than the required \emph{weakly} anti-commutation in the assumptions for the spectral triple.
    \item[4)] We argue that for an indefinite spectral triple, the correct assumption is for the real part and imaginary part of the Dirac operator to weakly \emph{anti-commute}, in both the even and odd case. This is a modification of the definition of odd indefinite spectral triples in \cite{dungen_rennie_2016}, where it is stated that in the odd case the real and imaginary part should weakly \emph{commute}. We support the claim that weakly anti-commuting is the correct notion by working out the example of $\mathrm{SU}(1,1)$, and we recover the doubling construction of \cite{dungen_rennie_2016} to obtain an even indefinite spectral triple from an odd spectral triple under the assumption of weakly anti-commuting real and imaginary part.
\end{enumerate}

Moreover, a Dirac operator on homogeneous spaces $G/K$ can be defined as simply the difference of the cubic Dirac operator on $G$ and on $K$. The hope is that knowledge of the spectral triple on $\mathrm{SU}(1,1)$ in this way can be used to obtain a spectral triple on the Poincaré disc $\mathrm{SU}(1,1)/\mathrm{U}(1)$. This direction will, however, not be pursued in this work.

The paper is structured as follows. In section \ref{sec_cubic_dirac} we fix notation for $\mathrm{SU}(1,1)$, and compute the cubic Dirac operator. Moreover, we discuss how this operator fits in the geometric formulation. Then, in section \ref{sec_harmonic_analysis} we explain the harmonic analysis of $\mathrm{SU}(1,1)$, which is the most important tool for the proofs in the later sections. In section \ref{sec_spectral_triples} we give the definition of a pseudo-Riemannian spectral triple and an indefinite spectral triple, and formulate the main theorems, which state that $\mathrm{SU}(1,1)$ together with the cubic Dirac operator forms both these spectral triples. All of section \ref{sec_proof} is devoted to proving these theorems.

\textbf{Acknowledgments.} I am grateful to Hessel Posthuma and Mikhail Isachenkov for their supervision and insightful feedback during this project. I also thank Koen van den Dungen and Adam Rennie for helpful correspondence concerning the two spectral triples treated in this paper. Finally, I am grateful to Teun van Nuland for valuable discussions regarding the proof of Lemma \ref{compact_resolvent}.

\section{The cubic Dirac operator on $\mathrm{SU}(1,1)$}\label{sec_cubic_dirac}
The Lie group $\mathrm{SU}(1,1)$ is the group of complex $2\times 2$-matrices of unit determinant that preserve the Hermitian form with signature $J=\begin{pmatrix} 1 & 0 \\ 0 & -1 \end{pmatrix}$, that is:
\begin{equation*}
    \mathrm{SU}(1,1)=\{ A\in M_2(\mathbb{C})\mid\det{A}=1, AJA^*=J\}.
\end{equation*}

The Lie algebra of $\mathrm{SU}(1,1)$ is the real Lie algebra $\mathfrak{su}(1,1)$ spanned by $\{e_1,e_2,e_3\}$, with Lie brackets $[e_1,e_2]=e_3$, $[e_2,e_3]=-e_1$, $[e_3,e_1]=-e_2$. The basis is orthogonal with respect to the Killing form $B(X,Y)\coloneq\Tr\ad X\ad Y$ on $\mathfrak{su}(1,1)$, which is symmetric, bilinear and non-degenerate because $\mathfrak{su}(1,1)$ is semisimple. One computes
\begin{equation*}
    B(e_1,e_1)=2,\qquad B(e_2,e_2)=2,\qquad B(e_3,e_3)=-2.
\end{equation*}
so that the $B$-dual basis is $\{e^1,e^2,e^3\}\coloneq\{\tfrac{1}{2}e_1,\tfrac{1}{2}e_2,-\tfrac{1}{2}e_3\}$. The universal enveloping algebra $\mathcal{U}(\mathfrak{su}(1,1))$ has center generated by the Casimir element $\Omega$ given by
\begin{equation*}
    \Omega=e_1e^1+e_2e^2+e_3e^3=\tfrac{1}{2}e_1^2+\tfrac{1}{2}e_2^2-\tfrac{1}{2}e_3^2.
\end{equation*}

The Clifford algebra $\Cl(\mathfrak{su}(1,1))$ is the associative algebra generated by the elements of $\mathfrak{su}(1,1)$, with relations
\begin{equation*}
    XY+YX=2B(X,Y),\qquad X,Y\in\mathfrak{su}(1,1).
\end{equation*}
These relations correspond to the convention used in \cite{meinrenken_2013}. The Clifford algebra is isomorphic as vector spaces to the exterior algebra $\wedge\mathfrak{su}(1,1)$. This isomorphism is explicitly realized by the quantization map $q\colon\wedge\mathfrak{su}(1,1)\rightarrow \Cl(\mathfrak{su}(1,1))$ (see e.g. \cite[prop. 2.7]{meinrenken_2013}), which for all $X_1,...,X_k\in\mathfrak{su}(1,1)$ is given by
\begin{equation*}
    q(X_1\wedge\cdots\wedge X_k)=\frac{1}{k!}\sum_{\sigma\in S_k}\sgn(\sigma)X_{\sigma_1}\cdots X_{\sigma_k},
\end{equation*}
where $S_k$ is the group of permutations of $1,...,k$ and $\sgn(\sigma)$ is the parity of a permutation $\sigma$.

Invariance of the Killing form implies that it is totally anti-symmetric, thus defining an element of $\wedge^3\mathfrak{su}(1,1)$ known as the structure constant tensor, which in terms of the (dual) basis is given by
\begin{equation*}
    \phi\coloneq-\frac{1}{12}\sum_{abc}B([e_a,e_b],e_c)e^a\wedge e^b\wedge e^c.
\end{equation*}
An explicit computation gives 
\begin{equation}
    \phi=-\tfrac{1}{8}e_1\wedge e_2\wedge e_3.
\end{equation}

Kostant and Sternberg \cite{kostant_sternberg_1987} were the first to observe that the image of the structure constant tensor $q(\phi)=-\tfrac{1}{8}e_1e_2e_3$ under the quantization map squares to a constant. This was used by Kostant to define a Dirac operator \cite{kostant_1999}, which he called the cubic Dirac operator, as an element of $\mathcal{U}(\mathfrak{su}(1,1))\otimes\Cl(\mathfrak{su}(1,1))$ of the following form
\begin{equation}\label{eq_cubic_dirac}
    D_0\coloneq\sum_{a}e^a\otimes e_a+1\otimes q(\phi).
\end{equation}
The cubic Dirac operator squares to
\begin{equation*}
    D_0^2=\Omega\otimes 1+\tfrac{1}{24}\Tr\Omega\otimes 1,
\end{equation*}
with the trace computed in the adjoint representation.

\begin{remark}
    In fact, the cubic Dirac operator defined above is a special case of the Dirac operator defined by Kostant in \cite{kostant_1999}. The more general setting considers a quadratic Lie algebra $\mathfrak{g}$ with $\mathfrak{k}\subset\mathfrak{g}$ a quadratic subalgebra. Let $\mathfrak{p}=\mathfrak{k}^\perp$ be the orthogonal complement of $\mathfrak{k}$ in $\mathfrak{g}$, so that $\mathfrak{g}=\mathfrak{k}\oplus\mathfrak{p}$. The relative Dirac operator is the element $D_{\mathfrak{g},\mathfrak{k}}\in (\mathcal{U}(\mathfrak{g})\otimes\Cl(\mathfrak{p}))^{\mathfrak{k}\text{-inv}}$ given by \cite[Proposition 7.7]{meinrenken_2013}
    \begin{equation*}
        D_{\mathfrak{g},\mathfrak{k}}=\sum_a{}^{(\mathfrak{p})} e_a\otimes e^a+q(\phi_\mathfrak{p}).
    \end{equation*}
    Here the sum $\sum_{a,b}{}^{(\mathfrak{p})}$ indicates summation over the basis of $\mathfrak{p}$, and $\phi_\mathfrak{p}$ is the structure constant tensor of $\mathfrak{p}$ (i.e., only summing over the basis of $\mathfrak{p}$).

    The relative Dirac operator is alternatively obtained as a difference of the cubic Dirac operator for $\mathfrak{g}$ and for $\mathfrak{k}$. For this, let $\lambda_\mathfrak{g}(\xi)\coloneq-\tfrac{1}{4}\sum_{a,b}B(\xi,[e_a,e_b])e^a\wedge e^b$. For $\xi\in\mathfrak{k}$, this decomposes as $\lambda_\mathfrak{g}(\xi)=\lambda_\mathfrak{k}(\xi)+\lambda_\mathfrak{p}(\xi)$ for some $\lambda_\mathfrak{p}(\xi)\in\wedge^2\mathfrak{p}$, by virtue of $[\mathfrak{k},\mathfrak{p}]\subset\mathfrak{p}$. Define an injective algebra homomorphism $j\colon \mathcal{U}(\mathfrak{k})\otimes\Cl(\mathfrak{k})\rightarrow\mathcal{U}(\mathfrak{g})\otimes\Cl(\mathfrak{g})$ by sending generators $K\otimes 1, 1\otimes K\in \mathcal{U}(\mathfrak{k})\otimes\Cl(\mathfrak{k})$ to the elements $K\otimes 1+1\otimes q(\lambda_\mathfrak{p}(K))$ and $1\otimes K$ in $\mathcal{U}(\mathfrak{g})\otimes\Cl(\mathfrak{g})$, respectively. Write $D_\mathfrak{g}$ for the cubic Dirac operator of $\mathfrak{g}$, and $D_\mathfrak{k}$ for the cubic Dirac operator of $\mathfrak{k}$, then $D_{\mathfrak{g},\mathfrak{k}}\coloneq D_\mathfrak{g}-j(D_{\mathfrak{k}})$. The relative Dirac operator squares to
    \begin{equation*}
        D_{\mathfrak{g},\mathfrak{k}}^2=\Omega_\mathfrak{g}-j(\Omega_\mathfrak{k})+\tfrac{1}{24}\Tr_\mathfrak{g}\Omega_\mathfrak{g}-\tfrac{1}{24}\Tr_\mathfrak{k}\Omega_\mathfrak{k},
    \end{equation*}
    with traces computed in the adjoint representation \cite[Theorem 2.13]{kostant_1999} \cite[Theorem 7.4]{meinrenken_2013}.
\end{remark}

We represent the Clifford algebra $\Cl(\mathfrak{su}(1,1))$ on $\mathbb{C}^2$ using the Pauli matrices, as
\begin{equation}\label{eq_clifford_rep}
    c(e_1)=\sqrt{2}\sigma_1,\qquad c(e_2)=\sqrt{2}\sigma_2,\qquad c(e_3)=\sqrt{2}i\sigma_3,
\end{equation}
where
\begin{equation}\label{eq_pauli_matrices}
    \sigma_1=\begin{pmatrix} 0 & 1 \\ 1 & 0 \end{pmatrix},\qquad \sigma_2=\begin{pmatrix} 0 & -i \\ i & 0 \end{pmatrix},\qquad \sigma_3=\begin{pmatrix} 1 & 0 \\ 0 & -1\end{pmatrix}.
\end{equation}
In this representation, the cubic Dirac operator \eqref{eq_cubic_dirac} takes the form
\begin{equation}\label{eq_cubic_dirac_rep}
    D_0=\frac{1}{\sqrt{2}}
    \begin{pmatrix} 
    -ie_3+\tfrac{1}{2} & e_1-ie_2 \\ 
    e_1+ie_2 & ie_3+\tfrac{1}{2} 
    \end{pmatrix}.
\end{equation}

The cubic Dirac operator acts on sections of the (canonical) spinor bundle $S=\mathrm{SU}(1,1)\times \mathbb{C}^2$, trivialized using right-translation. To understand this action, consider the left-regular representation of $\mathrm{SU}(1,1)$, which is an $\mathrm{SU}(1,1)$-action on a function $f\in C^\infty(\mathrm{SU}(1,1))$ defined by
\begin{equation*}
    (L_hf)(g)=f(h^{-1}g).
\end{equation*}
The corresponding action of the Lie algebra $\mathfrak{su}(1,1)$ is the action of a right-invariant vector field,  
\begin{equation*}
    (X\cdot f)(g)=\frac{d}{dt}\bigg|_{t=0}f\big(e^{-tX}g\big),
\end{equation*}
which extends uniquely to an action of the universal enveloping algebra $\mathcal{U}(\mathfrak{su}(1,1))$. Since sections of the spinor bundle are smooth functions $s\colon\mathrm{SU}(1,1)\rightarrow\mathbb{C}^2$, the Clifford algebra $\Cl(\mathfrak{su}(1,1))$ acts on $\Gamma(S)$ by Clifford multiplication through the representation \eqref{eq_clifford_rep}. Together, we obtain an action of $\mathcal{U}(\mathfrak{su}(1,1))\otimes\Cl(\mathfrak{su}(1,1))$ on $\Gamma(S)$.

\begin{remark}
    There is another spinor bundle on $\mathrm{SU}(1,1)$, which can be obtained from the universal covering manifold $\widetilde{\mathrm{SU}(1,1)}$. We, however, chose to only consider the canonical spinor bundle for this paper.
\end{remark}

\subsection{Geometric formulation}
Dirac operators can be constructed geometrically by lifting a metric-compatible connection on the tangent bundle to the spinor bundle. The cubic Dirac operator is also obtained in this way, as was first observed by Goette \cite{goette_1999} and Agricola \cite{agricola_2003}. One can find the construction in \cite[Chapter 9]{meinrenken_2013} for the case where the spinor bundle and tangent bundle are trivialized using left-translation. We briefly show how the result applies to our case, where the bundles are trivialized using right-translation.

There are two connections on the tangent bundle of $\mathrm{SU}(1,1)$ that are, in some sense, natural to consider. The first of these is the connection that is flat for right-invariant vector fields. This connection is called the natural connection $\nabla^\mathrm{nat}$, and for $X^R$ and $Y^R$ right-invariant vector fields, it is defined by
\begin{equation*}
    \nabla^\mathrm{nat}_{X^R}Y^R=0,
\end{equation*}
and has torsion $T(X^R,Y^R)=-[X^R,Y^R]$.
\begin{remark}
    The commutator of vector fields $[X^R,Y^R]$ is itself a right-invariant vector field. Using the BCH formula, one computes $[X^R,Y^R]=-[X,Y]^R$ with $[X,Y]$ the commutator in the Lie algebra. 
\end{remark}
The second connection is the Levi-Civita connection $\nabla^\mathrm{met}$, defined with respect to the (pseudo-Riemannian) metric
\begin{equation*}
    g(X,Y)(h)=-\tfrac{1}{2}B(dR_{h^{-1}}X(e),dR_{h^{-1}}Y(e)),
\end{equation*} 
for vector fields $X,Y\in\mathfrak{X}(\mathrm{SU}(1,1))$ and $h\in\mathrm{SU}(1,1)$. For right-invariant vector fields $X^R$ and $Y^R$, $\nabla^\mathrm{met}$ is given by \footnote{This expression for the Levi-Civita connection can be derived from the Koszul formula
\begin{align*}
    g(2\nabla_XY,Z)=&X(g(Y,Z))+Y(g(X,Z))-Z(g(X,Y))\\
    &+g([X,Y],Z)-g([X,Z],Y)-g([Y,Z],X).
\end{align*}}
\begin{equation*}
    \nabla^\mathrm{met}_{X^R}Y^R=\tfrac{1}{2}[X^R,Y^R].
\end{equation*}
These two connections define a one-parameter family of connections
\begin{equation*}
    \nabla^t=(1-2t)\nabla^\mathrm{nat}+2t\nabla^\mathrm{met}.
\end{equation*}
The cubic Dirac operator is the geometric Dirac operator for the connection $\nabla^t$ with $t=\tfrac{1}{3}$ \cite{agricola_2003}, \cite[Theorem 9.1]{meinrenken_2013}. The connections $\nabla^t$ determine a one-parameter family of Dirac operators $D_0^t$, given by
\begin{equation}\label{dirac_one_parameter}
    D_0^t=\sum_a e^a\otimes e_a +3t\otimes q(\phi).
\end{equation}

\section{Spectral decomposition}\label{sec_harmonic_analysis}
We set out to compute the spectrum of $D_0$, which requires some harmonic analysis of $\mathrm{SU}(1,1)$, in particular the Plancherel theorem for $\mathrm{SU}(1,1)$ and the matrix elements of the tempered representations. 

To this end, realize the Lie algebra $\mathfrak{su}(1,1)$ in terms of the Pauli matrices \eqref{eq_pauli_matrices} as $e_1=\tfrac{1}{2}\sigma_1$, $e_2=\tfrac{1}{2}\sigma_2$, $e_3=\tfrac{i}{2}\sigma_3$. Then any element of $g\in\mathrm{SU}(1,1)$ can be parametrized in terms of the Euler angles 
\begin{equation}\label{eq_parametrization_SU(1,1)}
    g(\varphi,t,\psi)=e^{\varphi e_3}e^{te_1}e^{\psi e_3},\qquad \varphi\in[0,2\pi),\quad \psi\in [-2\pi,2\pi),\quad t\geq 0.
\end{equation}
The Lie group $\mathrm{SU}(1,1)$ is unimodular, so it comes equipped with a left- and right-invariant Haar measure $dg$ that is unique up to scaling. In terms of the Euler angles, we fix the multiplicative constant and define the Haar measure by
\begin{equation}\label{eq_haar_measure}
    dg\coloneq\frac{1}{4\pi^2}\sinh t \ d\varphi \ dt \ d\psi.
\end{equation} 
With the Haar measure, define the space of square integrable functions $L^2(\mathrm{SU}(1,1))$ as the Hilbert space completion of the compactly supported functions $C^\infty_c(\mathrm{SU}(1,1))$ with the inner product
\begin{equation*}
    \langle f,h\rangle=\int_{\mathrm{SU}(1,1)}f(g)\overline{h(g)}\ dg.
\end{equation*}
The Lie algebra $\mathfrak{su}(1,1)$ acts on $L^2(\mathrm{SU}(1,1))$ as unbounded operators.

Let $\widehat{\mathrm{SU}(1,1)}$ denote the unitary dual of $\mathrm{SU}(1,1)$. As we will see shortly, the components of the cubic Dirac operator act nicely on the matrix elements of unitary irreducible representations adapted to a chosen elliptic subgroup of $\mathrm{SU}(1,1)$. To use this fact, we employ the Plancherel decomposition of $L^2(\mathrm{SU}(1,1))$. This expresses the square integrable functions as a direct integral of Hilbert spaces over $\widehat{\mathrm{SU}(1,1)}$, integrated with respect to the Plancherel measure (a famous result that was first obtained by Harish-Chandra \cite{harish-chandra_1952}). We set out to write this decomposition explicitly, following \cite{vilenkin_klimyk_1991}. 

The unitary irreducible representations in the support of the Plancherel measure are called the tempered representations. For $\mathrm{SU}(1,1)$, these are the discrete series representations $T_l^\pm$, $l=-\tfrac{1}{2},-1,-\tfrac{3}{2},...$ and the principal unitary series $T_\chi$, $\chi=(i\rho-\tfrac{1}{2},\epsilon)$, $\rho\in\mathbb{R}$, $\epsilon=0,\tfrac{1}{2}$. Note that we have adopted the labeling convention of \cite{vilenkin_klimyk_1991}. Denote the corresponding Hilbert spaces by $\mathcal{H}_l^\pm$ and $\mathcal{H}_\chi$, respectively. Choosing a basis $(v_n^{l,\pm})$ and $(v_n^\chi)$ for these Hilbert spaces, the matrix elements of the representations are defined by
\begin{equation}\label{eq_matrix_elements}
    \begin{aligned}
        &T_\chi(g)v_n^\chi=\sum_m t^\chi_{mn}(g)v^\chi_m,\\
        &T_l^\pm(g)v_n^{l,\pm}=\sum_m t^{l,\pm}_{mn}(g)v_m^{l,\pm}.
    \end{aligned}
\end{equation}

The specific bases $(v_n^{l,\pm})$ and $(v_n^\chi)$ we use are adapted to the elliptic subgroup $\{e^{se_3}\mid s\in\mathbb{R}\}$, and thus diagonalize the action of $e_3$. In this basis, the operators $ie_3$, $e_1-ie_2$, and $e_1+ie_2$ (appearing as components of $D_0$) act as ladder operators on the matrix elements \eqref{eq_matrix_elements} (see e.g. \cite{vilenkin_klimyk_1991}). For the principal unitary series, one finds 
\begin{equation}\label{eq_action_principal}
    \begin{aligned}
        -ie_3t^\chi_{mn}&=m't^\chi_{mn},\\
        (e_1-ie_2)t^\chi_{mn}&=-(\tau-m'+1)t^\chi_{m-1,n},\\
        (e_1+ie_2)t^\chi_{mn}&=-(\tau+m'+1)t^\chi_{m+1,n}.
    \end{aligned}
\end{equation}
Here $m,n\in\mathbb{Z}$, $m'=m+\epsilon$ and $\tau=i\rho-\tfrac{1}{2}$. Moreover, for the negative discrete series one finds
\begin{equation}\label{eq_action_negative_discrete}
    \begin{aligned}
        -ie_3t^{l,-}_{mn}&=m't^{l,-}_{mn},\\
        (e_1-ie_2)t^{l,-}_{mn}&=-\sqrt{(l-m'+1)(-l-m')}t^{l,-}_{m-1,n},\\
        (e_1+ie_2)t^{l,-}_{mn}&=\sqrt{(l-m')(-l-m'-1)}t^{l,-}_{m+1,n},
    \end{aligned}
\end{equation}
with $-\infty<m,n<l-\epsilon$, $m'=m+\epsilon$, and $\epsilon=\tfrac{1}{2}$ if $2l\in\mathbb{Z}_\mathrm{odd}$ and $\epsilon=0$ if $2l\in\mathbb{Z}_\mathrm{even}$. Finally, for the positive discrete series the action is
\begin{equation}\label{eq_action_positive_discrete}
    \begin{aligned}
        -ie_3t^{l,+}_{mn}&=m't^{l,+}_{mn}\\
        (e_1-ie_2)t^{l,+}_{mn}&=\sqrt{(l+m')(-l+m'-1)}t^{l,+}_{m-1,n}\\
        (e_1+ie_2)t^{l,+}_{mn}&=-\sqrt{(l+m'+1)(-l+m')}t^{l,+}_{m+1,n}
    \end{aligned}
\end{equation}
with $-l-\epsilon\leq m,n<\infty$, $m'=m+\epsilon$ and also $\epsilon=\tfrac{1}{2}$ if $2l\in\mathbb{Z}_\mathrm{odd}$ and $\epsilon=0$ if $2l\in\mathbb{Z}_\mathrm{even}$. To prove \eqref{eq_action_principal}-\eqref{eq_action_positive_discrete}, one realizes the representations as acting on a certain class of smooth functions on the circle (c.f. \cite[Section 6.4]{vilenkin_klimyk_1991}). The elliptic basis is then the usual basis $e^{-in\theta}$, and the matrix elements are Fourier coefficients. By expressing the matrix elements in terms of the Euler angels \eqref{eq_parametrization_SU(1,1)}, the action is computed via the infinitesimal action of the Lie algebra, which one obtains by computing $\frac{dg}{d\varphi}g^{-1}$, $\frac{dg}{dt}g^{-1}$ and $\frac{dg}{d\psi}g^{-1}$.

\subsection{The Plancherel decomposition}
Let $\pi$ be a unitary irreducible representation of $\mathrm{SU}(1,1)$. For $f\in L^1(\mathrm{SU}(1,1))$, its Fourier transform at $\pi$ is
\begin{equation}
    \hat{f}(\pi)\coloneq\int f(g)\pi(g^{-1})\ dg,
\end{equation}
with $dg$ the Haar measure \eqref{eq_haar_measure}. 

Recall that the two-sided regular representation $\tau$ on $L^2(\mathrm{SU}(1,1))$ is defined by $\tau(g_1,g_2)f(g)=f(g_2^{-1}gg_1)$, and that convolution of two integrable functions $f_1,f_2\in L^1(\mathrm{SU}(1,1))$ gives another integrable function $f_1*f_2(g)=\int f_1(h)f_2(h^{-1}g)dh$. Moreover, recall that a unitary representation $\pi\in\widehat{\mathrm{SU}(1,1)}$ determines a dual representation $\bar{\pi}$ on the dual space $\mathcal{H}_\pi'$ of $\mathcal{H}_\pi$ by $\bar{\pi}(g)=\pi(g^{-1})'$. Identifying the Hilbert space tensor product $\mathcal{H}_\pi\hat{\otimes}\mathcal{H}_{\bar{\pi}}$ with the space of Hilbert--Schmidt operators on $\mathcal{H}_\pi$, the Plancherel theorem is as follows.

\begin{theorem}[Plancherel theorem for $\mathrm{SU}(1,1)$ \cite{harish-chandra_1952}]\label{plancherel_theorem}
    The Fourier transform $f\rightarrow\hat{f}$ maps $L^1(\mathrm{SU}(1,1))\cap L^2(\mathrm{SU}(1,1))$ into $\int^\oplus\mathcal{H}_\pi\hat{\otimes}\mathcal{H}_{\bar{\pi}}d\mu(\pi)$, and extends to a unitary map from $L^2(\mathrm{SU}(1,1))$ onto $\int^\oplus\mathcal{H}_\pi\hat{\otimes}\mathcal{H}_{\bar{\pi}}d\mu(\pi)$ that intertwines the two-sided regular representation $\tau$ with $\int^\oplus \pi\hat{\otimes}\bar{\pi}d\mu(\pi)$. Here $d\mu(\pi)$ is the Plancherel measure
    \begin{align*}
        &d\mu(\pi)=d\mu(T_l^-)+d\mu(T_l^+)+d\mu(T_\chi),\\
        &d\mu(T_l^-)=d\mu(T_l^+)=(-l-\tfrac{1}{2})dl,\qquad d\mu(T_\chi)=\frac{1}{2}\sum_{\epsilon=0,\tfrac{1}{2}}\rho\tanh\pi(\rho+i\epsilon)d\rho,
    \end{align*}
    with $dl$ the counting measure on $\tfrac{1}{2}\mathbb{Z}_{<0}$ and $d\rho$ the Lebesgue measure on $\mathbb{R}$. For $\phi$ in the linear span of $\{f*g\mid f,g\in L^1\cap L^2\}$, one has the Fourier inversion formula
    \begin{equation*}
        \phi(g)=\int\Tr\big[\pi(g)\hat{\phi}(\pi)\big]d\mu(\pi).
    \end{equation*}
\end{theorem}
The Plancherel theorem, and especially the Fourier inversion formula, allows to write the infinitesimal action of the Lie algebra as a field of (unbounded) operators acting on the Plancherel decomposition of $L^2(\mathrm{SU}(1,1))$. 

Let us finish with a remark on a class of functions on $\mathrm{SU}(1,1)$, which are needed for the proof of condition $(4)$ of Definition \ref{pseudo-riemannian_spectral_triple} that is presented later. These functions are Harish-Chandra Schwartz functions $\mathcal{C}(\mathrm{SU}(1,1))$, which form a special class of functions lying densely in $L^2(\mathrm{SU}(1,1))$ that are similar to Schwartz functions on $\mathbb{R}^n$. In particular, they have fast decay properties, contain the smooth compactly supported functions $C^\infty_c(\mathrm{SU}(1,1))$, and the Fourier transform maps $\mathcal{C}(\mathrm{SU}(1,1))$ onto $\mathcal{C}(\widehat{\mathrm{SU}(1,1)})$, see e.g. \cite{wallach_1988}. 

\section{Spectral triples}\label{sec_spectral_triples}
Recall the representation $\mathbb{C}^2$, see \eqref{eq_clifford_rep}, of the Clifford algebra $\Cl(\mathrm{SU}(1,1))$. With the standard inner product on $\mathbb{C}^2$, we obtain the Hilbert space of $L^2$-spinors $L^2(\mathrm{SU}(1,1))\otimes\mathbb{C}^2$, with inner product of homogeneous elements $\varphi\otimes v$ and $\psi\otimes w$ given by
\begin{equation*}
    \langle\varphi\otimes v,\psi\otimes w\rangle=\langle \varphi,\psi\rangle_{L^2(\mathrm{SU}(1,1))}\langle v,w\rangle_{\mathbb{C}^2},
\end{equation*}
where $\langle\cdot,\cdot\rangle_\mathbb{C}$ is the standard inner product on $\mathbb{C}^2$, chosen to be anti-linear in the second argument.

We can compute the Hilbert space adjoint of the cubic Dirac operator using this inner product. For this, recall that the Dirac operator is $D_0=\sum_{a}e^a\otimes e_a+1\otimes q(\phi)$. Then for $\varphi,\psi\in C_c^\infty(\mathrm{SU}(1,1))$, using left-invariance of the Haar measure, we find that for $e^i$, $i=1,2,3$,
\begin{equation*}
    \langle e^i\cdot\varphi,\psi\rangle=\int \frac{d}{dt}\bigg|_{t=0}\varphi\big(e^{-te^i}g\big)\overline{\psi(g)}\ dg=\int\frac{d}{dt}\bigg|_{t=0}\varphi(g)\overline{\psi\big(e^{te^i}g\big)}\ dg=\langle\varphi,-e^i\cdot\psi\rangle.
\end{equation*}
Thus the Hilbert space adjoint $(e^i)^*$ extends $-e^i$, which is densely defined. Hence, we conclude $e^i$ is closable, and its adjoint is the closure of $-e^i$. The adjoint of the Clifford algebra generators can immediately be deduced from \eqref{eq_clifford_rep}, and we find $c(e_1)^*=c(e_1)$, $c(e_2)^*=c(e_2)$, and $c(e_3)^*=-c(e_3)$. It follows that $q(\phi)$ is self-adjoint. Combining the above, we find
\begin{equation*}
    D_0^*=\sum_a (e^a)^*\otimes c(e_a)^*+1\otimes q(\phi)^*
\end{equation*}
This has densely defined domain containing $C_c^\infty(\mathrm{SU}(1,1))\otimes \mathbb{C}^2$, so $D_0$ is closable and we denote its closure by $D\coloneq \overline{D_0}$. Moreover, its adjoint $D^*$ is the closure of
\begin{equation*}
    E_0=-e^1\otimes c(e_1)-e^2\otimes c(e_2)+e^3\otimes c(e_3)+1\otimes q(\phi).
\end{equation*}
The Hilbert space $L^2(\mathrm{SU}(1,1))\otimes\mathbb{C}^2$ and the Dirac operator $D$ are used in the spectral triples for $\mathrm{SU}(1,1)$.

\subsection{Pseudo-Riemannian spectral triple}\label{sec_pseudo_riemannian}
With the Dirac operator $D$, we define two second-order operators. Namely, let $\langle D\rangle^2\coloneq\tfrac{1}{2}(DD^*+D^*D)$ and $R_D\coloneq\tfrac{i}{2}(D^2-D^{*2})$. The operator $\langle D\rangle^2$ is positive, so it makes sense to define the set of smooth vectors as $\mathcal{H}_\infty\coloneq\bigcap_{k\geq 0}\dom\langle D\rangle^k$. Then for $T\in B(\mathcal{H})$ such that $T\colon\mathcal{H}_\infty\rightarrow\mathcal{H}_\infty$, let $\delta(T)\coloneq[(1+\langle D\rangle^2)^{1/2},T]$. The set of regular order-$r$ pseudodifferential operators is $\OP^r(\langle D\rangle)\coloneq(1+\langle D\rangle^2)^{r/2}\big(\bigcap_{n\in\mathbb{N}}\dom\delta^n\big)$.
\begin{definition}[\cite{dungen_paschke_rennie_2013}]\label{pseudo-riemannian_spectral_triple}
    A pseudo-Riemannian spectral triple $(\mathcal{A},\mathcal{H},D)$ consists of a $*$-algebra $\mathcal{A}$ represented on the Hilbert space $\mathcal{H}$ as bounded operators, along with a densely defined closed operator $D\colon\dom D\subset \mathcal{H}\rightarrow\mathcal{H}$ such that
    \begin{enumerate}[noitemsep]
        \item[(1)] $\dom D^*D\cap\dom DD^*$ is dense in $\mathcal{H}$ and $\langle D\rangle^2$ is essentially self-adjoint on this domain.
        \item[(2a)] $R_D\colon\mathcal{H}_\infty\rightarrow\mathcal{H}_\infty$, and $[\langle D\rangle^2,R_D]\in\OP^2(\langle D\rangle)$.
        \item[(2b)] $aR_D(1+\langle D\rangle^2)^{-1}$ is compact for all $a\in\mathcal{A}$.
        \item[(3)] $[D,a]$ and $[D^*,a]$ extend to bounded operators on $\mathcal{H}$ for all $a\in\mathcal{A}$.
        \item[(4)] $a(1+\langle D\rangle^2)^{-1/2}$ is compact for all $a\in\mathcal{A}$.
    \end{enumerate}
\end{definition}

One of our main results is that $\mathrm{SU}(1,1)$ with the cubic Dirac operator forms a pseudo-Riemannian spectral triple.

\begin{theorem}\label{spectral_triple}
    Let $C_c^\infty(\mathrm{SU}(1,1))$ act on $L^2(\mathrm{SU}(1,1))$ via pointwise multiplication. Then the triple $(\mathcal{A},\mathcal{H},D)\coloneq(C^\infty_c(\mathrm{SU}(1,1))\otimes 1,L^2(\mathrm{SU}(1,1))\otimes\mathbb{C}^2,D)$, with $D$ the closure of the cubic Dirac operator, is a pseudo-Riemannian spectral triple in the sense of Definition \ref{pseudo-riemannian_spectral_triple}.
\end{theorem}

The next section is devoted to the proof of this theorem. Let us, however, first look at some of its essential ingredients.

The two operators that are particularly important for a pseudo-Riemannian spectral triple are $\langle D\rangle^2=\tfrac{1}{2}(DD^*+D^*D)$ and $R_D=\tfrac{i}{2}(D^2-D^{*2})$. Let us compute these operators.

Observe that $D_0 E_0\subset DD^*$, and as $D_0E_0$ is defined on $\dom(D_0)$, this implies $DD^*$ is densely defined. Similarly, $D^*D$ is densely defined, with domain containing $\dom(D_0)$, and we conclude that $\dom(D^*D)\cap\dom(DD^*)$ is dense in $\mathcal{H}$. Moreover, $DD^*$ and $D^*D$ are self-adjoint and positive, so that $\langle D\rangle^2$ is symmetric and positive.

As $\langle D\rangle^2$ extends $\tfrac{1}{2}(D_0E_0+E_0D_0)$, on the domain of $D_0$ the action of $\langle D\rangle^2$ is given by
\begin{equation}
    \begin{aligned}
        &\tfrac{1}{2}(D_0E_0+E_0D_0)=\\
        &\frac{1}{2}\begin{pmatrix}
        (-ie_3+\tfrac{1}{2})^2-(e_1-ie_2)(e_1+ie_2) & 0 \\ 0 & (ie_3+\tfrac{1}{2})^2-(e_1+ie_2)(e_1-ie_2)
    \end{pmatrix}.
    \end{aligned}
\end{equation}
The action of the matrix components of $\langle D\rangle^2$ on the matrix elements of the principal unitary series and discrete series can easily be computed using the actions \eqref{eq_action_principal}-\eqref{eq_action_positive_discrete}. This gives
\begin{equation}\label{eq_action_first_component}
    \begin{aligned}
         \tfrac{1}{2}(-e_1^2-e_2^2-e_3^2-2ie_3+\tfrac{1}{4})t^\chi_{mn}&=(\tfrac{1}{2}|\tau|^2+m'(m'+1)+\tfrac{1}{8})t^\chi_{mn},\\
         \tfrac{1}{2}(-e_1^2-e_2^2-e_3^2-2ie_3+\tfrac{1}{4})t^{l,-}_{mn}&=(-\tfrac{1}{2}l(l+1)+m'(m'+1)+\tfrac{1}{8})t^{l,-}_{mn},\\
         \tfrac{1}{2}(-e_1^2-e_2^2-e_3^2-2ie_3+\tfrac{1}{4})t^{l,+}_{mn}&=(-\tfrac{1}{2}l(l+1)+m'(m'+1)+\tfrac{1}{8})t^{l,+}_{mn},
    \end{aligned}    
\end{equation}
and
\begin{equation}\label{eq_action_second_component}
    \begin{aligned}
         \tfrac{1}{2}(-e_1^2-e_2^2-e_3^2+2ie_3+\tfrac{1}{4})t^\chi_{mn}&=(\tfrac{1}{2}|\tau|^2+m'(m'-1)+\tfrac{1}{8})t^\chi_{mn},\\
         \tfrac{1}{2}(-e_1^2-e_2^2-e_3^2+2ie_3+\tfrac{1}{4})t^{l,-}_{mn}&=(-\tfrac{1}{2}l(l+1)+m'(m'-1)+\tfrac{1}{8})t^{l,-}_{mn},\\
         \tfrac{1}{2}(-e_1^2-e_2^2-e_3^2+2ie_3+\tfrac{1}{4})t^{l,+}_{mn}&=(-\tfrac{1}{2}l(l+1)+m'(m'-1)+\tfrac{1}{8})t^{l,+}_{mn}.
    \end{aligned}    
\end{equation}
Here $\tau=i\rho-\tfrac{1}{2}$, with $\rho\in\mathbb{R}$, so that $|\tau|^2=\rho^2+\tfrac{1}{4}$. Moreover, recall that $m'=m+\epsilon$, and that for the matrix elements of the principal unitary series $m,n$ can take any integer value, while for the negative discrete series $-\infty<m,n<l-\epsilon$ and for the positive discrete series $-l-\epsilon\leq m,n<\infty$. 

On the other hand, note that $R_D$ extends $\tfrac{1}{2}(D_0^2-E_0^2)$. One readily finds $E_0^2=D_0^2$, and as $D_0^2-E_0^2$ is defined on a dense domain we conclude $R_D=0$. 

\begin{remark}
    The fact that $R_D=0$ for the cubic Dirac operator is a big advantage, because it means that condition (2a) and (2b) of Definition \ref{pseudo-riemannian_spectral_triple} are automatically satisfied. It is interesting to note that for the one-parameter family of Dirac operators $D^t_0$ \eqref{dirac_one_parameter}, the operator $R_D$ does not vanish. Indeed, a quick computation gives that $\langle D^t\rangle^2$ extends
    \begin{equation}\label{one-parameter_positive_squared}
        \tfrac{1}{2}\big(D^t_0E^t_0+E^t_0D^t_0\big)=\Big(\tfrac{1}{2}(-e_1^2-e_2^2-e_3^2)+\tfrac{9}{8}t^2\Big)\otimes 1-\tfrac{1}{4}(1+3t)e_3\otimes c(e_1)c(e_2),
    \end{equation}
    while $R_{D^t}$ extends
    \begin{equation}\label{one-parameter_R}
        \tfrac{i}{2}\big((D^t_0)^2-(E^t_0)^2\big)=\tfrac{i}{4}(1-3t)\big(e_1\otimes c(e_2)c(e_3)+e_2\otimes c(e_1)c(e_3)\big).
    \end{equation}
\end{remark}

\subsection{Indefinite spectral triple}\label{sec_indefinite}
The first version of indefinite spectral triples \cite{dungen_rennie_2016} was not fully satisfactory, due to the fact that the main example of pseudo-Riemannian spin manifolds did not satisfy the assumptions. To be precise, in the first version it was required that the first-order operators $\Ree D\coloneq \tfrac{1}{2}(D+D^*)$ and $\Imm D\coloneq -\tfrac{i}{2}(D-D^*)$ almost anti-commute in a specific sense. For this to hold in the example, additional assumptions on the spacelike reflection $r$ had to be imposed. The definition was improved upon in \cite{dungen_2019}, where instead $\Ree D$ and $\Imm D$ are required to weakly anti-commute, made possible using results by Lesch and Mesland \cite{lesch_mesland_2019}. 

Let $S$ and $T$ be two closed and unbounded operators on a Hilbert space $\mathcal{H}$. Their anti-commutator $\{S,T\}=ST+TS$, is defined on the domain
\begin{equation*}
    \dom \{S,T\}=\{x\in \dom S\cap\dom T\mid Sx\in\dom T~\&~Tx\in \dom S\}.
\end{equation*}
If, moreover, $S$ and $T$ are self-adjoint, then they are called weakly anti-commuting if \cite[Definition 2.1]{lesch_mesland_2019}
\begin{enumerate}
    \item there is a constant $C>0$ such that for all $x\in\dom \{S,T\}$, we have
    \begin{equation*}
        \langle \{S,T\}x,\{S,T\}x\rangle\leq C(\langle x,x\rangle+\langle Sx,Sx\rangle+\langle Tx,Tx\rangle);
    \end{equation*}
    \item there is a $\lambda_0>0$ such that $(S+\lambda)^{-1}(\dom T)\subset\dom\{S,T\}$ for $\lambda\in i\mathbb{R}$, $|\lambda|\geq \lambda_0>0$.
\end{enumerate}
The operators $S$ and $T$ are said to weakly commute if the above requirements are satisfied with the anti-commutator replaced by the commutator.

\begin{definition}[\cite{dungen_2019}]\label{indefinite_spectral_triple}
    An indefinite spectral triple $(\mathcal{A},\mathcal{H},D)$ consists of a $*$-algebra $\mathcal{A}$ represented on the Hilbert space $\mathcal{H}$ as bounded operators, together with a densely defined closed operator $D\colon \dom D\subset\mathcal{H}\rightarrow\mathcal{H}$, such that
    \begin{enumerate}[noitemsep]
        \item $\Ree D$ and $\Imm D$ are essentially self-adjoint on $\dom D$.
        \item The pair $(\Ree D,\Imm D)$ weakly anti-commutes.
        \item $[D,a]$ and $[D^*,a]$ extend to bounded operators on $\mathcal{H}$ for all $a\in\mathcal{A}$.
        \item $a(1+DD^*+D^*D)^{-1}$ is compact for all $a\in\mathcal{A}$.
    \end{enumerate}
\end{definition}

\begin{remark}
    In fact, Definition \ref{indefinite_spectral_triple} describes an odd indefinite spectral triple. The spectral triple is said to be even if there exists an operator $\Gamma\in B(\mathcal{H})$ such that $\Gamma=\Gamma^*$, $\Gamma^2=1$, $\Gamma a=a\Gamma$ for all $a\in\mathcal{A}$ and $\Gamma D+D\Gamma=0$. Note that in \cite{dungen_rennie_2016} it is stated that an odd indefinite spectral triple should have the pair $(\Ree D,\Imm D)$ weakly \emph{commuting} instead of $(\Ree D,\Imm D)$ weakly \emph{anti-commuting} (or, in the setting of \cite{dungen_rennie_2016}, almost commuting instead of almost anti-commuting). However, we suggest that the modification of the definition as shown above is better, being compatible with the case of odd-dimensional pseudo-Riemannian manifolds. Also in the more general setting of unbounded Kasparov modules the correct definition is with weakly anti-commuting operators, as opposed to weakly commuting operators. The doubling procedure to obtain an even indefinite Kasparov module from an odd indefinite Kasparov module with the definition above is more straightforward than the procedure in \cite{dungen_rennie_2016}, and uses the operator
    \begin{equation*}
        \widetilde{D}=\begin{pmatrix}
            0 & D \\
            D & 0
        \end{pmatrix},
    \end{equation*}
    acting on the $\mathbb{Z}_2$ graded Hilbert space $\mathcal{H}\oplus\mathcal{H}$.
\end{remark}

One of our main results is that $\mathrm{SU}(1,1)$ with the cubic Dirac operators forms an indefinite spectral triple.

\begin{theorem}\label{spectral_triple_2}
    Let $C_c^\infty(\mathrm{SU}(1,1))$ act on $L^2(\mathrm{SU}(1,1))$ via pointwise multiplication. Then the triple $(\mathcal{A},\mathcal{H},D)\coloneq(C^\infty_c(\mathrm{SU}(1,1))\otimes 1,L^2(\mathrm{SU}(1,1))\otimes\mathbb{C}^2,D)$, with $D$ the closure of the cubic Dirac operator, is an indefinite spectral triple in the sense of Definition \ref{indefinite_spectral_triple}.
\end{theorem}
The next section is devoted to the proof of this theorem, and of Theorem \ref{spectral_triple}. But we first consider the first-order operators $\Ree D$ and $\Imm D$ in a bit more detail.

Recall that the real part of $D$ is $\Ree D=\tfrac{1}{2}(D+D^*)$, and the imaginary part is $\Imm D=-\tfrac{i}{2}(D-D^*)$. These operators extend $\tfrac{1}{2}(D_0+E_0)$ and $-\tfrac{i}{2}(D_0-E_0)$, respectively. In the representation \eqref{eq_clifford_rep}, we have
\begin{equation}\label{real_and_imaginary}
    \begin{aligned}
        \tfrac{1}{2}(D_0+E_0)&=\frac{1}{\sqrt{2}}\begin{pmatrix}
            -ie_3+\tfrac{1}{2} & 0 \\
            0 & ie_3+\tfrac{1}{2}
        \end{pmatrix},\\
        -\tfrac{i}{2}(D_0-E_0)&=-\frac{i}{\sqrt{2}}\begin{pmatrix}
            0 & e_1-ie_2 \\
            e_1+ie_2 & 0
        \end{pmatrix}.
    \end{aligned}
\end{equation}
The anti-commutator $\{\Ree D,\Imm D\}$ extends the anti-commutator of the above operators. A simple computation shows that $\{\Ree D,\Imm D\}=0$, as
\begin{equation}\label{eq_anticommutator}
    \{\tfrac{1}{2}(D_0+E_0),-\tfrac{i}{2}(D_0-E_0)\}=0.
\end{equation}

\begin{remark}
    Like for the pseudo-Riemannian spectral triple, also for the indefinite spectral triple working with the cubic Dirac operator is a big advantage. Indeed, the cubic Dirac operator makes it so that the anti-commutator $\{\Ree D,\Imm D\}$ vanishes. If instead we consider the one-parameter family of Dirac operators $D_0^t$ \eqref{dirac_one_parameter}, then for the anti-commutator we find
    \begin{equation*}
        \{\tfrac{1}{2}(D_0^t+E_0^t),-\tfrac{i}{2}(D_0^t-E_0^t)\}=\begin{pmatrix}
            0 & (3t-1)(e_1-ie_2)\\
            (3t-1)(e_1+ie_2) & 0
        \end{pmatrix},
    \end{equation*}
    which is a nonvanishing first-order differential operator, except for $t=\tfrac{1}{3}$.
\end{remark}

\begin{remark}
    To emphasize once more that weakly \emph{anti-commuting} is the correct notion for odd indefinite spectral triples (Definition \ref{indefinite_spectral_triple}), note that the commutator $[\Ree D,\Imm D]$ extends the operator
    \begin{equation*}
        [\tfrac{1}{2}(D_0+E_0),-\tfrac{i}{2}(D_0-E_0)]=-\frac{i}{2}\begin{pmatrix}
            0 & \{e_1-ie_2,-ie_3\}\\
            -\{e_1+ie_2,-ie_3\} & 0
        \end{pmatrix}.
    \end{equation*}
    This is a second-order operator, so that $\Ree D$ and $\Imm D$ cannot be weakly commuting. Indeed, it is not too difficult to construct a sequence of functions that make clear that no constant $C>0$ exists for which
    \begin{equation*}
        \langle [\Ree D,\Imm D]x,[\Ree D,\Imm D]x\rangle\leq C(\langle x,x\rangle+\langle \Ree Dx,\Ree Dx\rangle+\langle \Imm Dx,\Imm Dx\rangle).
    \end{equation*}
    For instance, take a sequence $(x_{-m})$ of matrix elements of the negative discrete series with $l=1$; $x_{-m}=\begin{pmatrix}
        t^{1,-}_{mn} \\ t^{1,-}_{m+1,n}
    \end{pmatrix}$, with $-m=2,3,4,...$. Then
    \begin{align*}
        \langle [\Ree D,\Imm D]x_{-m},&[\Ree D,\Imm D]x_{-m}\rangle\\
        &=-\frac{1}{4}(-2m-1)^2(1-m)(-m-2)\langle x_{-m},x_{-m}\rangle,
    \end{align*}
    which scales with $m^4$, while
    \begin{align*}
        \langle x_{-m},x_{-m}\rangle+\langle\Ree D x_{-m},&\Ree D x_{-m}\rangle+\langle\Imm D x_{-m},\Imm D x_{-m}\rangle\\
        &=\big(1+\tfrac{1}{2}(m+\tfrac{1}{2})^2-\tfrac{1}{2}(1-m)(-m-2)\big)\langle x_{-m},x_{-m}\rangle,
    \end{align*}
    which scales (at most) with $m^2$.
\end{remark}
    
\section{Proofs}\label{sec_proof}
We now turn to the proof of Theorem \ref{spectral_triple} and Theorem \ref{spectral_triple_2}. For both, the compact resolvent condition takes the most effort to prove. We check the other requirements case by case.

\subsection{Pseudo-Riemannian spectral triple} For the pseudo-Riemannian spectral triple in Theorem \ref{spectral_triple}, recall that the operator $R_D$ vanishes, as noted in section \ref{sec_pseudo_riemannian}. Hence, it remains to show that $\langle D\rangle^2$ is essentially self-adjoint, that $[D,a]$ and $[D^*,a]$ extend to bounded operators on $\mathcal{H}$ for all $a\in\mathcal{A}$, and that the weighted resolvent $a(1+\langle D\rangle^2)^{-1/2}$ is compact for all $a\in\mathcal{A}$.

\begin{lemma}\label{essentially_self-adjoint}
    The unbounded operator $\langle D\rangle^2=\tfrac{1}{2}(DD^*+D^*D)$ is essentially self-adjoint on $\dom(D^*D)\cap\dom(DD^*)$.
\end{lemma}
\begin{proof}
    From the action of the components on the matrix elements of the tempered representations \eqref{eq_action_first_component} and \eqref{eq_action_second_component}, it follows that $\langle D\rangle^2$ acts fiberwise on the direct integral decomposition of $L^2(\mathrm{SU}(1,1))\otimes \mathbb{C}^2$ obtained from Theorem \ref{plancherel_theorem}, with diagonal action on each fiber. Positivity of $\langle D\rangle^2$ implies that on each fiber the spectrum is contained in the interval $[0,\infty)$. Hence, the spectrum of $\langle D\rangle^2+1$ on each fiber is contained in the interval $[1,\infty)$. This implies that the range of $\langle D\rangle^2+1$ is dense, and we conclude that $\langle D\rangle^2$ is essentially self-adjoint on $\dom D^*D\cap\dom DD^*$.
\end{proof}

\begin{lemma}\label{bounded_commutator}
    All $a\in\mathcal{A}$ preserve $\dom D$ and $\dom D^*$, and both $[D,a]$ and $[D^*,a]$ extend to bounded operators on $\mathcal{H}$ for all $a\in\mathcal{A}$.
\end{lemma}
\begin{proof}
    It is clear that the domain of $D_0$ and $E_0$ is preserved under the action of $\mathcal{A}$. To show that $a$ preserves $\dom D$, let $w\in\dom D$. Then there is a sequence $(v_n)\subset\dom(D_0)$ such that $v_n\rightarrow w$ and $D_0 v_n\rightarrow Dw$. As $a$ is a bounded operator, also $av_n\rightarrow aw$. In order for $aw$ to be an element of $\dom D$, we need $D_0av_n\rightarrow v$ for some $v\in \mathcal{H}$; then $v=Daw$ as $Da$ is a closed operator. We check this explicitly using the expression \eqref{eq_cubic_dirac} for the cubic Dirac operator. Let $a=f\otimes 1$ for some $f\in C_c^\infty(\mathrm{SU}(1,1)$, then using the Leibniz rule we find
    \begin{equation*}
        D_0av_n=D_0fv_n=fD_0v_n+\bigg(\sum_i e^i(f)\otimes c(e_i)\bigg)v_n.
    \end{equation*}
    The expression between brackets is a bounded operator on $\mathcal{H}$, by virtue of $f$ being smooth and compactly supported. Thus, using convergence of $D_0v_n$, we conclude that $a$ preserves $\dom D$. Moreover, from the above computation, we find for the commutator $[D,a]=\sum_i e^i(f)\otimes c(e_i)$ which extends to a bounded commutator on $\mathcal{H}$.

    A similar computation shows that $\mathcal{A}$ preserves $\dom D^*$ and that $[D^*,a]$ extends to a bounded operator on $\mathcal{H}$ for all $a\in\mathcal{A}$.
\end{proof}

We postpone the proof of the compact resolvent for now, and first consider the requirements for the indefinite spectral triple.

\subsection{Indefinite spectral triple} Recall that the anti-commutator $\{\Ree D,\Imm D\}$ vanishes, as noted in section \ref{sec_indefinite}. Moreover, Lemma \ref{bounded_commutator} applies for this case too. Hence, to prove Theorem \ref{spectral_triple_2}, it remains to show that $\Ree D$ and $\Imm D$ are essentially self-adjoint on $\dom D$ and that the resolvent $a(1+DD^*+D^*D)^{-1}$ is compact for all $a\in\mathcal{A}$. 

\begin{lemma}\label{real_essentially_self-adjoint}
    The unbounded operator $\Ree D$ is essentially self-adjoint on $\dom D$.
\end{lemma}
\begin{proof}
    The proof is similar to that of Lemma \ref{essentially_self-adjoint}. With the difference that, in this case, from the actions \eqref{eq_action_principal}-\eqref{eq_action_positive_discrete}, it follows that the spectrum on each fiber is real. Hence, the range of $\Ree D\pm i$ is dense, and we conclude that $\Ree D$ is essentially self-adjoint on $\dom D$.
\end{proof}

To show that $\Imm D$ is essentially self-adjoint, we use the following fact, which is stated in \cite[Problem X.28]{reed_simon_vol2}.
\begin{proposition}\label{squared_self_adjoint}
    Let $T$ be a symmetric operator and suppose the domain $\dom T^2$ of $T^2$ is dense. If $T^2$ is essentially self-adjoint on $\dom T^2$, then $T$ is essentially self-adjoint.
\end{proposition}

\begin{lemma}\label{imaginary_essentially_self_adjoint}
    The unbounded operator $\Imm D$ is essentially self-adjoint on $\dom D$.
\end{lemma}
\begin{proof}
    Note that $(\Imm D)^2$ extends $\big(-\tfrac{i}{2}(D_0-E_0)\big)^2$, which has dense domain $\dom D_0$. Hence, if we can show that $(\Imm D)^2$ is essentially self-adjoint on $\dom D^2$, by Proposition \ref{squared_self_adjoint}, it follows that $\Imm D$ is essentially self-adjoint on $\dom D$. From \eqref{real_and_imaginary}, it follows that $(\Imm D)^2$ extends
    \begin{equation*}
        -\frac{1}{2}\begin{pmatrix}
            (e_1-ie_2)(e_1+ie_2) & 0 \\ 0 & (e_1+ie_2)(e_1-ie_2)
        \end{pmatrix}.
    \end{equation*}
    Note that on the matrix elements of the principal unitary series,
    \begin{align*}
        (e_1-ie_2)(e_1+ie_2) t^\chi_{mn}&=-(|\tau|^2+m'(m'+1))t^\chi_{mn},\\
        (e_1+ie_2)(e_1-ie_2)t^\chi_{mn}&=-(|\tau|^2+m'(m'-1))t^\chi_{mn},
    \end{align*}
    and that similarly the action on the matrix elements of the discrete series is diagonal, with real coefficients. Following the proof of Lemma \ref{real_essentially_self-adjoint}, it follows that $(\Imm D)^2$ is essentially self-adjoint, and thus that $\Imm D$ is essentially self-adjoint.
\end{proof}

\subsection{The weighted resolvent}
We will show that the weighted resolvent for the pseudo-Riemannian spectral triple is compact, this can be used to show compactness of the weighted resolvent for the indefinite spectral triple as well. 

As $\langle D\rangle^2$ is essentially self-adjoint and positive, the operator $\langle D\rangle^2+1$ is invertible with positive bounded inverse $(\langle D\rangle^2+1)^{-1}$. Its square root $R\coloneq(\langle D\rangle^2+1)^{-1/2}$ is called the resolvent of $\langle D\rangle$. For $(\mathcal{A},\mathcal{H},D)$ to be a pseudo-Riemannian spectral triple, it is required for the weighted resolvent $aR$ to be compact for all $a\in\mathcal{A}$. To show this, we write $aR^p$ as an integral operator, and show that for $p$ big enough, this operator is Hilbert--Schmidt. Then, using the following preliminary result, we can conclude that $aR$ is compact itself:

\begin{lemma}\label{compactness_powers}
    Let $A$ be a bounded operator on $\mathcal{H}$ and let $B$ be bounded and self-adjoint. If $AB^q$ is compact for some $q\in\mathbb{R}_{>0}$, then it is compact for all $q>0$.
\end{lemma}
\begin{proof}
    Suppose there is a $q_0>0$ for which $AB^{q_0}$ is compact. If $q>q_0$, then $AB^q=AB^{q_0}B^{q-q_0}$ is compact because $B^{q-q_0}$ is a bounded operator and the compact operators form an ideal. 

    On the other hand, if $q<q_0$, note that $AB^{q_0}A^*$ is compact and of the form $CC^*$ for $C=AB^{q_0/2}$, so that, using the functional calculus for $z\mapsto\sqrt{z}$ and the polar decomposition, also $AB^{q_0/2}$ is compact. Repeating this, eventually we obtain that $AB^{q_0/2^n}$ is compact for some natural number $n$ for which $q_0/2^n<q$. Hence, also $AB^q$ is compact.
\end{proof}

Note that $R$ can be written as a diagonal matrix 
\begin{equation}
    R=\begin{pmatrix} R_+ & 0 \\ 0 & R_- \end{pmatrix}=\begin{pmatrix}R_+ & 0 \\ 0 & 0 \end{pmatrix}+\begin{pmatrix} 0 & 0 \\ 0 & R_- \end{pmatrix},
\end{equation}
with $R_\pm=\Big(-\tfrac{1}{2}(e_1^2+e_2^2+e_3^2)\mp ie_3+\tfrac{9}{8}\Big)^{-1/2}$. Write $a=f\otimes 1$; since the compact operators form an ideal, $aR$ is compact if both $fR_+$ and $fR_-$ are compact (acting on $L^2(\mathrm{SU}(1,1))$).

The operators $R_\pm$ act as a field of operators $\pi\rightarrow R_\pm(\pi)$ on the Plancherel decomposition, and on each Hilbert space in the decomposition the action is diagonal. That is, if we let $v_n^\chi$ and $v^{l,\pm}_n$ be the elliptic bases of before, then the action of $R_\pm(\pi)$ is
\begin{equation}\label{eq_action_R}
    \begin{aligned}
        R_\pm(T^\chi)&=\sum_m\Big(\tfrac{1}{2}|\tau|^2+m'(m'\pm1)+\tfrac{9}{8}\Big)^{-1/2} v^\chi_m\otimes (v^\chi_m)^*,\\
        R_\pm(T^{l,-})&=\sum_m \Big(-\tfrac{1}{2}l(l+1)+m'(m'\pm1)+\tfrac{9}{8}\Big)^{-1/2} v_m^{l,-}\otimes (v_m^{l,-})^*,\\
        R_\pm(T^{l,+})&=\sum_m \Big(-\tfrac{1}{2}l(l+1)+m'(m'\pm1)+\tfrac{9}{8}\Big)^{-1/2} v_m^{l,+}\otimes (v_m^{l,+})^*.
    \end{aligned}
\end{equation}
Here the second tensor legs contain elements of the dual basis. 

Let $\varphi\in\mathcal{C}(\mathrm{SU}(1,1))$ be a Harish-Chandra Schwartz function on $\mathrm{SU}(1,1)$. For $f\in C_c^\infty(\mathrm{SU}(1,1))$, the action of $fR_+$ on $\varphi$ is
\begin{equation}
    \begin{aligned}
        (fR_+\varphi)(g)&=\int f(g)\Tr\big[R_+(\pi)\pi(g)\hat{\varphi}(\pi)\big]d\mu(\pi)\\
        &=\int f(g)\Tr\big[R_+(\pi)\pi(g)\int\varphi(h)\pi(h^{-1})dh \big]d\mu(\pi) 
    \end{aligned}
\end{equation}
Raising $R_+$ to a power $p>0$, we aim to rewrite this to an integral operator of the form $\int K(g,h)\varphi(h)dh$, with $K\in L^2(\mathrm{SU}(1,1))\hat{\otimes} L^2(\mathrm{SU}(1,1))$, which is Hilbert--Schmidt and thus compact. For this we need some results regarding convergence of infinite sums.

\begin{proposition}[\cite{siegel_1980}]\label{sum_convergence}
    Let $Q(x,y)=ax^2+by^2$ be a quadratic form with $a,b>0$, then the series $\sum_{\mathbb{Z}^2\setminus 0}\frac{n}{Q(m,n)^{q/2}}$ is absolutely convergent for $q>3$.
\end{proposition}
A proof of this was originally obtained by Siegel \cite{siegel_1980}, see also \cite[Theorem 1.3]{hiroe_oda_2008}.

\begin{lemma}\label{series_1}
    The series
    \begin{equation*}
        S^-_\pm\coloneq\sum_{l\leq-\tfrac{1}{2}}\Big(-l-\tfrac{1}{2}\Big)\sum_m \frac{1}{\Big(-\tfrac{1}{2}l(l+1)+m'(m'\pm 1)+\tfrac{9}{8}\Big)^{q/2}}
    \end{equation*}
    with $l\in\tfrac{1}{2}\mathbb{Z}_{<0}$, $m<l-\epsilon$, $m'=m+\epsilon$ and $\epsilon=0$ when $l\in\mathbb{Z}$, while $\epsilon=\tfrac{1}{2}$ when $l\in\tfrac{1}{2}\mathbb{Z}_\mathrm{odd}$, converges for $q>3$. 
\end{lemma}
\begin{proof}
   Replace $m$ by defining $-k=m+\epsilon-l$, and $l$ by defining $n=-2l-1$ to obtain $S_{\pm}^-$ as a double sum over the natural numbers. In this form, one immediately obtains the series inequalities 
    \begin{equation*}
        S_1\coloneq \sum_{n=1}^\infty \sum_{k=1}^\infty\frac{\tfrac{1}{2}n}{\Big(4k^2+\tfrac{25}{8}n^2\Big)^{q/2}} \leq S^-_\pm\leq \sum_{n=1}^\infty \sum_{k=1}^\infty\frac{\tfrac{1}{2}n}{\Big(k^2+\tfrac{1}{8}n^2\Big)^{q/2}}\eqcolon S_2.
    \end{equation*}
   Define the quadratic forms $Q_1(x,y)=4x^2+\tfrac{25}{8}y^2$ and $Q_2(x,y)=x^2+\tfrac{1}{8}y^2$, to write $S_i=\tfrac{1}{2}\sum_{n=1}^\infty\sum_{k=1}^\infty\frac{n}{Q_i(k,n)^{q/2}}$ for $i=1,2$. Now note that we can rewrite
    \begin{equation*}
        \sum_{\mathbb{Z}^2\setminus 0}\bigg|\frac{n}{Q_i(k,n)^{q/2}}\bigg|=8S_i+2\sum_{n=1}^\infty\frac{n}{Q_i(0,n)^{q/2}}.
    \end{equation*}
   The left hand side is absolutely summable for $q>3$ by Proposition \ref{sum_convergence}. Moreover, the sum $\sum_{n=1}^\infty \frac{n}{Q_i(0,n)^{q/2}}$ is well-known and converges for $q>2$. Hence, we conclude that both series $S_i$, $i=1,2$, converge for $q>3$ and thus that the series $S^-_\pm$ converges for $q>3$.   
\end{proof}

\begin{lemma}\label{series_2}
    The series
    \begin{equation*}
        S^+_\pm\coloneq\sum_{l\leq-\tfrac{1}{2}}\Big(-l-\tfrac{1}{2}\Big)\sum_m\frac{1}{\Big(-\tfrac{1}{2}l(l+1)+m'(m'\pm 1)+\tfrac{9}{8}\Big)^{q/2}}
    \end{equation*}
    with $m\geq-l-\epsilon$ and $l,m',\epsilon$ as in Lemma \ref{series_1}, converges for $q>3$.
\end{lemma}
\begin{proof}
    Replace $m$ by defining $k=m+l+\epsilon+1$, and replace $l$ by defining $n=-2l-1$. We then get two sums over the natural numbers, and obtain the series inequalities
    \begin{equation*}
            S_1\coloneq\sum_{n=1}^\infty\sum_{k=1}^\infty\frac{\tfrac{1}{2}n}{\Big(2k^2+\tfrac{9}{8}n^2\Big)^{q/2}}\leq S^+_\pm \leq \sum_{n=1}^\infty\sum_{k=1}^\infty\frac{\tfrac{1}{2}n}{\Big(\tfrac{1}{2}k^2+\tfrac{1}{8}n^2\Big)^{q/2}}\eqcolon S_2.
    \end{equation*}
    From here, similar to the proof of Lemma \ref{series_1}, it follows that the series $S^+_\pm$ converges for $q>3$.
\end{proof}

\begin{lemma}\label{series_3}
    The series
    \begin{equation*}
        S^\chi_\pm=\sum_{\epsilon=0,\tfrac{1}{2}}\int_0^\infty \sum_m\frac{1}{\Big(\tfrac{1}{2}\rho^2+m'(m'\pm1)+\tfrac{5}{4}\Big)^{q/2}}\rho\tanh\pi(\rho+i\epsilon)d\rho
    \end{equation*}
    with $m\in\mathbb{Z}$ and $m'=m+\epsilon$ converges for $q>3$.
\end{lemma}
\begin{proof}
    Writing out the sum over $\epsilon$, we find $S^{\chi}_\pm=S^{\chi,0}_\pm+S^{\chi,1/2}_\pm$ with
    \begin{align*}
        S^{\chi,0}_\pm&\coloneq\int_0^\infty\sum_{m\in\mathbb{Z}}\frac{\rho\tanh\pi\rho}{\Big(\tfrac{1}{2}\rho^2+m(m+1)+\tfrac{5}{4}\Big)^{q/2}}d\rho,\\
        S^{\chi,1/2}_\pm&\coloneq\int_0^\infty\sum_{m\in\mathbb{Z}}\frac{\rho\coth\pi\rho}{\Big(\tfrac{1}{2}\rho^2+(m+1)^2+1\Big)^{q/2}}d\rho.
    \end{align*}
    Let us first consider $S^{\chi,0}_\pm$. Aiming to apply Fubini's theorem, we compute
    \begin{equation*}
        \tilde{S}^{\chi,0}_\pm\coloneq\sum_{m\in\mathbb{Z}}\int_0^\infty \frac{\rho\tanh\pi\rho}{\Big(\tfrac{1}{2}\rho^2+m(m+1)+\tfrac{5}{4}\Big)^{q/2}}d\rho.
    \end{equation*}
    Note that by approximating the hyperbolic tangent, we have the inequalities
    \begin{align*}
        \sum_{m\in\mathbb{Z}}\int_1^\infty&\frac{\tfrac{1}{2}\rho}{\Big(\tfrac{1}{2}\rho^2+m(m+1)+\tfrac{5}{4}\Big)^{q/2}}d\rho\leq \\
        &\tilde{S}^{\chi,0}_\pm\leq \sum_{m\in\mathbb{Z}}\int_0^\infty \frac{\rho}{\Big(\tfrac{1}{2}\rho^2+m(m+1)+\tfrac{5}{4}\Big)^{q/2}}d\rho
    \end{align*}
    The integrals are solved explicitly using substitution, and converge for $q>2$, which gives
    \begin{equation*}
            \frac{1}{q-2}\sum_{m\in\mathbb{Z}} \frac{1}{\Big(m(m+1)+\tfrac{7}{4}\Big)^{q/2-1}}\leq \tilde{S}^{\chi,0}_\pm\leq \frac{1}{q/2-1}\sum_{m\in\mathbb{Z}} \frac{1}{\Big(m(m+1)+\tfrac{5}{4}\Big)^{q/2-1}}.
    \end{equation*}
    Now the remaining summand on both sides is quadratic in $m$, and thus converges for $q>3$.

    For $S^{\chi,1/2}_\pm$, note that by estimating the hyporbolic cotangent, we get the inequalities
    \begin{align*}
        \sum_{m\in\mathbb{Z}}&\int_0^\infty \frac{\rho}{\Big(\tfrac{1}{2}\rho^2+(m+1)^2+1\Big)^{q/2}}d\rho\leq\sum_{m\in\mathbb{Z}}\int_0^\infty \frac{\rho\coth\pi\rho}{\Big(\tfrac{1}{2}\rho^2+(m+1)^2+1\Big)^{q/2}}d\rho.\\
        &\leq\sum_{m\in\mathbb{Z}}\int_0^1 \frac{1}{\Big(\tfrac{1}{2}\rho^2+(m+1)^2+1\Big)^{q/2}}d\rho+\sum_{m\in\mathbb{Z}}\int_1^\infty \frac{2\rho}{\Big(\tfrac{1}{2}\rho^2+(m+1)^2+1\Big)^{q/2}}d\rho
    \end{align*}
    A similar method shows that the bounds converge for $q>3$. An application of Fubini's theorem finishes the proof.
\end{proof}

With the above results in hand, we can prove the main result of this section.

\begin{lemma}\label{compact_resolvent}
    Let $a\in\mathcal{A}$ and $R=(\langle D\rangle^2+1)^{-1/2}$ the resolvent of $\langle D\rangle$. Then $aR$ is a compact operator on $\mathcal{H}$.
\end{lemma}
\begin{proof}
    We show that $aR^q$ is Hilbert--Schmidt for $q>3$, which using Lemma \ref{compactness_powers} implies compactness of $aR$. Write $a=f\otimes 1$ for $f\in C_c^\infty(\mathrm{SU}(1,1))$; then $aR^q$ is Hilbert--Schmidt if both $fR_+^q$ and $fR_-^q$ are Hilbert--Schmidt operators on $L^2(\mathrm{SU}(1,1))$. We show that $fR_+^q$ is Hilbert--Schmidt for $q>3$. 

    Let $\varphi\in\mathcal{C}(\mathrm{SU}(1,1))$, then $\hat{\varphi}\in\mathcal{C}(\widehat{\mathrm{SU}(1,1)})$ is trace class, so it has a well-defined inverse Fourier transform. Using \eqref{eq_action_R}, the action of $fR_+^q$ is
    \begin{equation*}
        fR_+^q\varphi(g)=\int f(g)\Tr\bigg[R_+^q(\pi)\pi(g)\int\varphi(h)\pi(h^{-1})dh\bigg]d\mu(\pi).
    \end{equation*}
    To write $fR_+^q$ as an integral operator, we need to swap the order of integration. Note that $R_+^q$ is trace class for $q>1$, as is easily checked from \eqref{eq_action_R}, so that we can rewrite
    \begin{equation*}
        fR_+^q\varphi(g)=\int f(g) \int\Tr\big[R_+^q(\pi)\pi(gh^{-1})\big]\varphi(h)dhd\mu(\pi).
    \end{equation*}
    To swap the two integrals, we apply Fubini's theorem. As $g$ is fixed, we consider
    \begin{equation*}
        \int\int\Big|\Tr\big[R_+^q(\pi)\pi(gh^{-1})\big]\varphi(h)\Big|dhd\mu(\pi)\leq \int\Tr\big[R_+^q(\pi)\big]d\mu(\pi)\int|\varphi(h)|dh
    \end{equation*}
    Writing out the trace and the Plancherel measure gives
    \begin{equation*}
        \int\Tr\big[R_+^q(\pi)\big]d\mu(\pi)=S^+_++S^-_++S^\chi_+,
    \end{equation*}
    which converges by Lemmas \ref{series_1}, \ref{series_2}, and \ref{series_3}, as we assumed $q>3$. Hence, by Fubini's theorem, we obtain
    \begin{equation*}
        fR_+^q\varphi(g)=\int f(g)\int\Tr\big[R_+^q(\pi)\pi(gh^{-1})\big]d\mu(\pi)\varphi(h)dh\eqcolon\int K(g,h)\varphi(h)dh.
    \end{equation*}

    From Lemmas \ref{series_1}-\ref{series_3}, it follows that $R_+^q\in L^2(\widehat{\mathrm{SU}(1,1)})$. Hence, using the Plancherel theorem, $\int\Tr\big[R_+^q(\pi)\pi(gh^{-1})\big]d\mu(\pi)$ is an element of $L^2(\mathrm{SU}(1,1))$ (in the variable $h$). As $f\in L^2(\mathrm{SU}(1,1))$, we conclude that $K_+(g,h)\in L^2(\mathrm{SU}(1,1)\times \mathrm{SU}(1,1))$.
    
    For $\psi\in L^2(\mathrm{SU}(1,1))$, let $(\varphi_m)\subset\mathcal{C}(\mathrm{SU}(1,1))$ be a sequence converging to $\psi$ in the $L^2$-norm, then also $fR_+^q \varphi_m\rightarrow fR_+^q\psi$ in $L^2(\mathrm{SU}(1,1))$. There is a subsequence, also denoted $\varphi_m$, that converges pointwise almost everywhere. Using that $K_+$ is an $L^2$-kernel, we find
    \begin{equation*}
        \lim_{m\rightarrow\infty}(fR_+^q\varphi_m)(g)=\lim_{m\rightarrow\infty}\int K_+(g,h)\varphi_m(h)dh=\int K_+(g,h)\psi(h)dh,
    \end{equation*}
    which we use to write
    \begin{equation*}
        fR_+^q\psi(g)=\int K_+(g,h)\psi(h)dh.
    \end{equation*}
    Thus, $fR_+^q$ is an integral operator with kernel in $L^2(\mathrm{SU}(1,1)\times \mathrm{SU}(1,1))$, so that $fR_+^q$ is Hilbert--Schmidt. As the proof for $fR_-^q$ is similar, this concludes the proof.
\end{proof}

\begin{remark}\label{rmk_proof_abelian}
    The proof of Lemma \ref{compact_resolvent} is similar to that of (the combination of)  \cite[Theorem IX.29]{reed_simon_vol2} and \cite[Theorem XI.20]{reed_simon_vol3} 
\end{remark}

The hard work for the proofs of Theorem \ref{spectral_triple} and \ref{spectral_triple_2} has been done now, what remains is a matter of combining the results.

\begin{proof}[Proof of Theorem \ref{spectral_triple}]
    We check the conditions of Definition \ref{pseudo-riemannian_spectral_triple}. For condition (1), $\langle D\rangle^2$ is essentially self-adjoint on $\operatorname{dom} D^*D\cap\operatorname{dom}DD^*$ by Lemma \ref{essentially_self-adjoint}. Condition (2) is automatically fulfilled, as $R_D=0$. Condition (3) is proven in Lemma \ref{bounded_commutator}. Finally, condition (4) is proven in Lemma \ref{compact_resolvent}.
\end{proof}

\begin{proof}[Proof of Theorem \ref{spectral_triple_2}]
    By Lemma \ref{real_essentially_self-adjoint} and \ref{imaginary_essentially_self_adjoint}, it follows that condition (1) of Definition \ref{indefinite_spectral_triple} is satisfied. Condition (2) follows because $\{\Ree D,\Imm D\}=0$. Condition (3) follows from Lemma \ref{bounded_commutator}. Condition (4) is proven similar to Lemma \ref{compact_resolvent}, using that $a(1+DD^*+D^*D)^{-1}=a(1+2\langle D\rangle^2)^{-1}$. 
\end{proof}

\begin{remark}
    The proofs of condition (1), (3) and (4) of Definition \ref{pseudo-riemannian_spectral_triple} as written above for the cubic Dirac operator $D$ go through for the one-parameter family $D^t$ \eqref{dirac_one_parameter}, up to small adjustments for the parameter $t$. Especially the series in Lemmas \ref{series_1}-\ref{series_3} converge, when adjusted for $\langle D^t\rangle^2$ (as in \eqref{one-parameter_positive_squared}). However, to conclude that $(\mathcal{A},\mathcal{H},D^t)$ is a pseudo-Riemannian spectral triple, requires one to prove conditions (2a) and (2b) for $R_{D^t}$ \eqref{one-parameter_R}, which is nontrivial if $t\neq \tfrac{1}{3}$. We will not pursue this proof here.

    Similarly, the proofs of condition (1), (3) and (4) of Definition \ref{indefinite_spectral_triple} work for the one-parameter family $D^t$. However, we will not prove that the pair $(\Ree D^t,\Imm D^t)$ weakly anticommutes for $t\neq \tfrac{1}{3}$.
\end{remark}

\end{document}